\documentclass{amsart}
\usepackage{amsmath,amssymb,amsthm}
\pdfoutput=1
\usepackage[
    pdftitle={On commutative weak BCK-algebras},
    pdfauthor={J. C\={\i}rulis}
]{hyperref}

\numberwithin{equation}{section}

\theoremstyle{plain}
\newtheorem{theo}{Theorem} [section]
\newtheorem{prop}[theo]{Proposition}
\newtheorem{lemm}[theo]{Lemma}
\newtheorem{coro}[theo]{Corollary}

\theoremstyle{definition}
\newtheorem{exam}{Example}
\newtheorem{defi}[theo]{Definition} 

\newcommand\colo{\colon\,}

\newcommand\Iff\Leftrightarrow

\newcommand\sminus\smallsetminus

\newcommand{\Nv}{\mathbf N^1_5}
\newcommand{\Nw}{\mathbf N^2_5}

\newcommand{\Ov}{\mathbf {OM}_6}

\renewcommand\phi\varphi
\newcommand\sqsub\sqsubset
\newcommand\di{{:}}
\newcommand{\tri}{{\therefore}}

\newcommand\thitem[1]{\item[\hspace*{-3ex}{\upshape (#1)}]}

\newenvironment{myenum}{\begin{enumerate}

        \setlength{\itemsep}{0pt}}
        {\end{enumerate}}

\newcounter{minus}
\newenvironment{minusenum}{\begin{enumerate}
        \setlength{\itemindent}{1.94em}
        \setcounter{enumi}{\value{minus}}

        \setlength{\itemsep}{0pt}}
{\setcounter{minus}{\value{enumi}}\end{enumerate}}
\newcommand{\rminus}[1]{($-_{\ref{#1}}$)}

\begin{document}
\setcounter{page}{1}

\title{On commutative  weak BCK-algebras}
\author{J\=anis C\={\i}rulis} 

\newcommand{\acr}{\newline\indent}
\address{Faculty of Computing\acr
    University of Latvia\acr
    Rai\c{n}a b., 19\acr
    Riga LV-1586\acr
    LATVIA}
\email{jc@lanet.lv}

\thanks{This work was supported by ESF project
No.\ 2009/0216/1DP1.1.1.2.0/09/APIA/VIAA /044}


\begin{abstract}
The class of weak BCK-algebras can be obtained by replacing (in the standard axiom set by K.~Iseki and S.~Tanaka) the first BCK axiom $(x - y) - (x - z) \le z - y$ by its weakening $z \le y \Rightarrow x - y \le x - z$.
It is known that every weak BCK-algebra is completely determined by the structure of its initial segments. We review several natural classes of commutative weak BCK-algebras, prove that they are equationally definable, and show that the order duals of a multitude of algebras with implication known in the literature in connection with various quantum logics are, in fact, commutative weak BCK-algebras belonging to that or other of these classes.
We also characterize initial segments of algebras in each of the classes as lattices equipped with a suitable kind of complementation. In particular, commutative weak BCK-algebras are just those meet semilattices with the least element in which all initial segments are non-distributive De Morgan lattices.
\end{abstract}

\maketitle

\section{Introduction}   \label{intro}

Like BCK-algebras \cite{bck1,bck2}, a weak BCK-algebra is a poset having the bottom element and equipped with a binary operation considered as subtraction. Both BCK-algebras and weak BCK-algebras have also been treated in the dual form as algebras with the reversed ordering, the top element and an operation considered as implication. We hold here the former viewpoint, and call algebras presented in the latter form BCK*-algebras, resp., weak BCK*-algebras.

The class of weak BCK*-algebras was introduced by the author in the talk on the 44-th Summer School on General Algebra and Ordered Sets (Rad\u{e}jov, Czech Republic, September 2006) as a tool for algebraization of posets with terminal sections equipped with certain weak complementations. They were studied more extensively in \cite{wbck1} and latter, already as algebras with subtraction, in \cite{wbck2,wbck3}. In contrast to these papers, the present one is devoted just to weak BCK-algebras. A representation theorem for a subclass of weak BCK*-algebras, which relates it to substructural logic, was proved in \cite{qbck}; see subsection \ref{prelim}.2 below.

\begin{defi}    \label{wBCK:def}
A \emph{weak BCK-algebra} (\emph{wBCK-algebra}, for short) is an algebra $(A,-,0)$, where $A$ is a poset with $0$ the least element, and $-$ is a binary operation on $A$ (called \emph{subtraction}) satisfying the axioms
\begin{minusenum}
\item $x \le y$ if and only if $x - y = 0$,    \label{salg}
\item if $x - y \le z$, then $x - z \le y$. \label{sgalois}
\end{minusenum}
\end{defi}

For instance, every poset with 0 carries a \emph{discrete wBCK-algebra}, where
\[
\mbox{$x - y = 0$ if $x \le y$, and $x - y = x$ otherwise}.
\]

The operation $-$ induces on every initial segment (or section) $[0,p]$ of a wBCK-algebra $A$ a unary operation $^+_p$ defined by $x^+_p := p - x$, which can be considered as a kind of complementation on $[0,p]$, named in \cite{wbck2} a g*-complementation. It is shown there that every wBCK-algebra is completely determined by the structure of its sections; even more, there is bijective connection between wBCK-algebras and certain sectionally g*-complemented posets with zero (see below Section \ref{struct} for details). Also, connections between several classes of wBCK-algebras and properties of respective sectional g*-complementations were found out in \cite{wbck2}; see also \cite{amst}.

In the present paper we continue this line of investigation, but deal mainly with the narrower class of commutative wBCK-algebras (i.e., those fulfilling the identity $x - (x - y) = y - (y - x)$ and a few of its subclasses. Sections \ref{prelim} and  \ref{commw} review some basic facts concerning general wBCK-algebras and, respectively, commutative wBCK-algebras. In particular, like BCK-algebras, every commutative wBCK-algebra is a meet semilattice, and all these algebras form a variety. Probably, the most familiar subclass of commutative BCK-algebras is that of implicative BCK-algebras; their order duals  are in fact equivalent (see, e.g., \cite{Jie}) to implication, or semi-boolean, algebras introduced earlier in \cite{Abb1a,Abb1b}. Weak BCK-algebras admit several reasonable generalizations of implicative BCK-algebras.
In Section \ref{oimpl}, isolated is a subclass of  \emph{orthoimplicative wBCK-algebras}, which satisfy the condition $x- y \le y \Rightarrow x \le y$. This class also is a variety. Further, in Section 5, orthoimplicative wBCK-algebras satisfying the condition $x \le y \le z \Rightarrow y - z \le x - z$, called \emph{implicative wBCK-algebras}, are shown to form another variety. Every orthoimplicative BCK-algebra is an implicative BCK algebra, and every implicative BCK-algebra is an implicative weak BCK-algebra.

We also disclose in Sections 3--5 that a number of algebras with implication known in the literature 
(see \cite{Abb2,Ch4,Ch2,Ch5,ChH1,ChH2,ChHK2,ChHK3,ChHL1,ChHL2,ChKS,ChL,MP})
 are, in fact, commutative wBCK*-algebras of that or other type.

At last, in Section \ref{struct}, we characterize the commutative, orthoimplicative and implicative wBCK-algebras in terms of sectional g*-complementations $^+_p$. In particular, initial sections in such algebras are non-distributive De Morgan lattices, ortholattices and orthomodular lattices, respectively. (Recall that the sections in implicative BCK-algebras are Boolean lattices.) Some results obtained in this section have been announced in \cite{amst}. Similar structure theorems for several algebras with implication are obtained also in the papers cited in the preceding paragraph; due to the unifying framework of wBCK-algebras, our proofs are more straightforward and simpler.

\cite{stlog} is a recent shortened version of this paper, which includes its Introduction, subsection 2.1, and reorganized Section \ref{struct} (with a few new results).

\section{Weak BCK-algebras} \label{prelim}

\subsection{Preliminaries} \label{prelim1}

To reduce uses of parentheses and thus make the structure of expressions more transparent, we follow \cite{wbck2} and use dots for grouping. Our convention for reading expressions is as follows: a left (right) delimiter in an expression is either a left (resp., right) round bracket, or a dot or group of dots which precedes (resp., follows) a term in this expression, or the beginning (resp., end) of the expression, and the intended scope of a binary operation is determined by the nearest left and the nearest right delimiter.
(However, a dot (group of dots) may delimit a scope only if the number of dots in this delimiter exceeds the number of not bracketed occurrences of operations in the corresponding wing of the scope.)
For instance, both expressions
\begin{gather*} x - y. - .z - x\di - .(x - \di y - .z -x) - y, \\
 (x - y. - .z - x) - (x - \di y - .z -x\tri - y)
 \end{gather*}
are condensed versions of \( ((x - y) - (z - x)) - ((x - (y - (z - x))) - y) \).

The subsequent list of basic properties of subtraction is borrowed from \cite{wbck1,wbck2}.

 \begin{prop} In any wBCK-algebra,   \label{wBCK-prop}
 \begin{minusenum}
\item $x - .x - y \le y$, \label{mdi2}
\item if $x \le y$, then $z - y \le z - x$, \label{manti}
\item $x - \di x - .x - y = x - y$, \label{mtri}
\item $x - x = 0$, \label{xx0}
\item $x - y \le x$, \label{xyx}
\item $x - .x - y \le x$, \label{mdi1}
\item $x - 0 = x$, \label{x0x}
\item $0 - x = 0\,$. \label{0x0}
\end{minusenum}
\end{prop}

It is easily seen that the axiom \rminus{sgalois} can be derived from \rminus{mdi2} and \rminus{manti}. Moreover, the axiom \rminus{salg} can be derived using \rminus{sgalois} and \rminus{x0x}: $x - y \le 0$ iff $x - 0 \le y$ iff $x \le y$.
These observations allow us to simplify the axioms of  wBCK-algebras.

\begin{prop}[\cite{wbck2}]    \label{wbck:prop}
A poset with a least element $0$ and a binary operation $-$ is a weak BCK-algebra if and only if it fulfills \textup{\rminus{mdi2}}, \textup{\rminus{manti} and \rminus{x0x}}.
\end{prop}

We have included the order relation $\le$ among the primitives to get a more concise axiom set. Of course, the equivalence \rminus{salg} may be considered as a definition of the relation in terms of subtraction and $0$, and then
the assumption that $A$ is a poset can be replaced by explicit order axioms \rminus{0x0}, \rminus{xx0} and
\begin{equation}
\label{leanti}
\mbox{if $x - y = 0$ and $y - x = 0$, then $x = y$}.
\end{equation}
The transitivity axiom,
\[
\mbox{if $x - y = 0$ and $y - z  = 0$, then $x - z = 0$},
\]
is derivable in thus reorganized axiom system by a use of \rminus{manti}, \rminus{0x0} and (\ref{leanti}): if $y - z = 0$, then $x - z. - .x - y = 0$, and if, in addition, $x - y = 0$, then $x - z. - 0 = 0 = 0 - .x - z$. So, $x - z = 0$.

\begin{coro} \label{wbck:eq}
An algebra $(A,-,0)$ is a wBCK-algebra with respect to the relation $\le$ defined on it by \textup{\rminus{salg}} if and only if it satisfies \textup{\rminus{mdi2}, \rminus{manti}, \rminus{xx0}, \rminus{0x0}} and \textup{(\ref{leanti})}.
\end{coro}

We note without proof that, due to \rminus{mdi2}, the axioms  \rminus{xx0} and \rminus{0x0} here may be replaced by one identity
\begin{minusenum}
\item \label{x.0yx}
$x - .0 - y = x$.
\end{minusenum}

BCK-algebras are defined in \cite{bck1,bck2} as algebras $(A,0,-)$ satisfying
the same axioms (listed in the corollary) with the exception of \rminus{manti}, which actually replaces here the (stronger) BCK-axiom
\begin{minusenum}
\item \label{mAnti}
$z - y. - .z - x \le x - y$.
\end{minusenum}
The weak BCK-algebras indeed form a wider class of algebras.

\begin{exam} \label{exN5}
The five-element poset with two maximal chains $0 < a < 1$ and $0 < b < c < 1$, i.e., the non-distributive lattice $\mathbf N_5$, provides an example of a wBCK-algebra with the following operation table for $-$:
 {\small
\[
\begin{array}{|c|@{\extracolsep{6pt}}ccccc|}
\hline
\;-\-\, &  0 &  a &  b &  c &  1\vphantom{)^{A^A}} \\
\hline
0 & 0 & 0 & 0 & 0 & 0\vphantom{)^{A^A}} \\
a & a & 0 & a & a & 0 \\
b & b & b & 0 & 0 & 0 \\
c & c & c & b & 0 & 0 \\
1 & 1 & a & c & b & 0 \\
\hline
\end{array}
\] }
(It is worth to note that the axiom \rminus{salg} together with \rminus{x0x} and the following consequence of \rminus{xyx}:
\begin{equation} \label{h1}
\mbox{$x - y = x$ whenever  $x$ is an atom and $x \nleq y$}
\end{equation}
completely determine most entries in the table; they likewise help in the further examples.  See also the more general condition (\ref{h2}).)

Let as denote by $\Nv$ the obtained wBCK-lattice. Now, the values $x = a$, $y = c$, $z = 1$ falsify the axiom \rminus{mAnti}; so $\Nv$ is not a BCK-algebra.
\end{exam}

It is known well that the class of all BCK algebras is not a variety \cite{Wr}. As noticed in \cite{wbck1}, since it consists of those wBCK-algebras satisfying the inequality \rminus{mAnti}, which can be rewritten as an equation, it follows that wBCK-algebras also do not form a variety. The powerful axiom \rminus{mAnti} can be split into a pair of weaker ones.

\begin{theo} \label{bck(w)}
A wBCK-algebra is a BCK-algebra if and only if it satisfies the conditions
\begin{minusenum}
\item \label{miso}
if $x \le y$, then $x - z \le y - z$,
\item \label{exch}
$x - y. - z = x - z. - y$.
\end{minusenum}
\end{theo}
\begin{proof}
Both these conditions are BCK-theorems---see (3) and (7) in \cite{bck2}. Conversely, \rminus{manti} and \rminus{miso} imply that $x - .x - y\di - z \le y - z$, from where \rminus{mAnti} follows by \rminus{exch}.
\end{proof}

\subsection{Quasi-BCK-algebras}

By a \emph{pocrig} (in full, a partially ordered  commutative residuated integral groupoid),  we mean here an algebra $(A,+,-,0)$, where $(A,+)$ is a partially ordered commutative groupoid, $0$  is its least and simultaneously its neutral element, and $-$ is a binary operation on $A$ characterized by the condition
\[
\mbox{$x \le y + z$ if and only if $x - y \le z$}.
\]
Associative pocrigs (i.e., monoids) are known as pocrims; BCK-algebras are just $(-,0)$-subreducts of pocrims (see, e.g., \cite[Section 2]{BR} and references therein).
The following analogue of this classical result was announced, in the dual form, in \cite{amst}, and proved in \cite{qbck} (Theorem 2).

\begin{prop} \label{repr-w}
An algebra $(A,-, 0)$ is a subreduct of a pocrig if and only if it is a wBCK-algebra and satisfies the isotonicity law \textup{\rminus{miso}}.
\end{prop}

Following \cite{qbck}, we call a weak BCK-algebra satisfying \rminus{miso} a \emph{quasi-BCK-algebra}, or just \emph{qBCK-algebra}. The class of all such algebras also is not a variety.

Due to the aforementioned representation theorem of qBCK-algebras, these algebras are considered as algebraic models of a substructural implicational logic (viz., the implicational fragment of the  associative Lambek calculus) with the rules of weakening and exchange, but without the contraction rule. Proposition \ref{repr-w} shows that  quasi-BCK-algebras may likewise be related with the non-associative substructural logic.

\subsection{Meets in wBCK-algebras}

If, in some wBCK-algebra, elements $x$ and $y$ have the meet $x \wedge y$, then
 \begin{equation}   \label{minus}
x - y = x - .x \wedge y \, .
\end{equation}
Indeed, $x - .x - y \le x \wedge y$ by \rminus{mdi1} and \rminus{mdi2}; then \rminus{sgalois} gives us the inequality $x - .x \wedge y \le x - y$. On the other hand, $x - y \le x - .x \wedge y$ by \rminus{manti}.
The following useful strengthened version of (\ref{h1}):
\begin{equation}  \label{h2}
\mbox{$x - y = x$ whenever $x \wedge y = 0$}
\end{equation}
is an immediate consequence of (\ref{minus}). 

Even if the meet of $x$ and $y$ does not exist, the element $x - y$ can be presented in a form $x - z$ with $z \le x,y$: see Proposition \ref{wBCK-prop}. Moreover, the following generalization of (\ref{minus}) holds in every wBCK-algebra due to \rminus{mdi2}--\rminus{mtri} and \rminus{mdi1}:
\begin{equation}   \label{Minus}
x - y = \min(x - z\colo z \le x,y).
\end{equation}
We may conclude that every wBCK-algebra is completely determined by the structure of its initial segments. See also subsection \ref{struct1}.

Most of wBCK-algebras we shall deal with in the further sections will be meet semilattices.
By a \emph{wBCK$^\wedge$-algebra} we mean an algebra $(A,\wedge,-,0)$, where $(A,\wedge)$ is a meet semilattice and $(A,-,0)$ is a wBCK-algebra w.r.t.\ its order $\le$. It follows from Theorem 4  of \cite{wbck1} that the class of all such algebras is equationally definable by semilattice axioms, \rminus{mdi2}, \rminus{x0x} and identities
\[
x \wedge y. - x = 0 \ \text{ and } \ z - y \le z - .x \wedge y \enspace.
\]
Evidently, the first of these two conditions can be eliminated in favor of \rminus{xx0}. Again, \rminus{xx0} and \rminus{x0x} may be replaced by \rminus{x.0yx}.

\subsection{Around positive implicative wBCK-algebras}

We adapt a term used for BCK-algebras \cite{bck1,bck2} and call a wBCK-algebra \emph{positive implicative} if it satisfies the \emph{contraction law}
\begin{minusenum}
\item \label{mcontr}
$x - y. - y = x - y$.
\end{minusenum}
A weaker form of the law is the \emph{contraction rule}
\begin{minusenum}
\item \label{mcontr-}
if $x - y \le y$, then $x \le y$.
\end{minusenum}
As noticed in \cite[p.\ 482]{wbck1}, the contraction rule is equivalent to \rminus{mcontr} in every BCK-algebra; however, it is not the case for weak BCK-algebras.

\begin{exam}    \label{posim}
Let us consider the poset depicted below at the left:

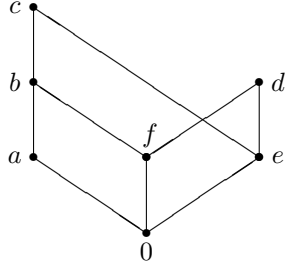
\setlength\unitlength{.5cm}
\begin{center}
\begin{picture}(8,7)(3,2.4)
\put(10,5){\line(0,1){2}}
\put(10,7){\line(-3,-2){3}}
\put(7,5){\line(0,-1){2}}
\put(7,3){\line(3,2){3}}
\put(10,5){\line(-3,2){6}}
\put(4,9){\line(0,-1){4}}
\put(4,5){\line(3,-2){3}}
\put(7,5){\line(-3,2){3}}
\put(10,5){\circle*{.2}}
\put(10,7){\circle*{.2}}
\put(7,5){\circle*{.2}}
\put(7,3){\circle*{.2}}
\put(10,5){\circle*{.2}}
\put(4,9){\circle*{.2}}
\put(4,7){\circle*{.2}}
\put(4,5){\circle*{.2}}
\put(7,5){\circle*{.2}}
\put(7,2.5){\makebox(0,0){$0$}}
\put(3.5,5){\makebox(0,0){$a$}}
\put(3.5,7){\makebox(0,0){$b$}}
\put(3.5,9){\makebox(0,0){$c$}}
\put(10.5,7){\makebox(0,0){$d$}}
\put(10.5,5){\makebox(0,0){$e$}}
\put(7.1,5.6){\makebox(0,0){$f$}}
\end{picture}
\qquad\qquad
\raisebox{1.8cm}{{\small \(
\begin{array}{|c|@{\extracolsep{6pt}}ccccccc|}
\hline
\;-\-\, &  0 &  a &  b &  c & d & e & f \vphantom{)^{A^A}} \\
\hline
0 & 0 & 0 & 0 & 0 & 0 & 0 & 0\vphantom{)^{A^A}} \\
a & a & 0 & 0 & 0 & a & a & a \\
b & b & f & 0 & 0 & a & b & a \\
c & c & c & e & 0 & b & c & c \\
d & d & d & d & d & 0 & d & d \\
e & e & e & e & 0 & 0 & 0 & e \\
f & f & f & 0 & 0 & 0 & f & 0 \\
\hline
\end{array}
\)} }
\end{center}

\noindent
When equipped with an operation $-$ with the presented table,
it becomes a wBCK-algebra (even a qBCK-algebra), in which \rminus{mcontr-} holds.  However, the algebra is not positive implicative: $c - d \neq c - d. - d$ (and, hence, not a BCK-algebra).
\end{exam}

The right distributive law
\begin{minusenum}
\item  \label{mdistr}
$x - z. - .y - z = x - y. -z$,
\end{minusenum}
which implies \rminus{mcontr}, and is equivalent to it in BCK-algebras \cite[Theorem 8]{bck2}, is too strong in the context of wBCK-algebras---it follows from Corollary 2.3 in \cite{ChH3} that a wBCK-algebra satisfying this law is necessarily a BCK-algebra.
A direct proof: by \rminus{xyx}, \rminus{manti} and \rminus{mdistr}, $z - y. - x \le z - y. - .x - y = z - x. - y$; this yields \rminus{exch}, and if $x \le y$, then by \rminus{salg}, \rminus{0x0} and \rminus{mdistr}, $0 = x - y = x - y. - z = x - z. - .y - z$; this yields \rminus{miso}.

In contrary, the \emph{Pierce law}
\begin{minusenum}
\item \label{pierce}
$x - .y - x = x$,
\end{minusenum}
which was used in \cite{bck2} as the defining condition for implicative BCK-algebras, turns out to be not strong enough  in the context of wBCK-algebras; we introduce implicative wBCK-algebras in another way in Section \ref{uoimpl}. Nevertheless, a wBCK-algebra satisfying \rminus{pierce} is always positive implicative: applying this identity twice, we get $x - y. - y = x - y. - \di y - .x - y = x - y$. Therefore, \rminus{mcontr-} also holds in such an algebra.
However, the three-element discrete wBCK-chain is an example of a positive implicative wBCK-algebra where \rminus{pierce} fails. We shall see in Section \ref{oimpl} that wBCK-algebras satisfying \rminus{pierce} form an equational class.

A wBCK-algebra satisfies the Pierce law if and only if the following consequence (by \rminus{salg}) of the law:
\begin{minusenum}
\item if $x \le y - x$, then $x = 0$. \label{pierce'}
\end{minusenum}
is fulfilled in it for all $x$ and $y$. Indeed, if \rminus{pierce'} holds, then, in particular, if $u \le a$ and $u \le b - a$ for some $a$ and $b$, then $b - a \le b - u$ by \rminus{manti} and, further, $u = 0$. Therefore,
\begin{equation}
x \wedge (y - x) = 0 \label{wm0}
\end{equation}
for all $x, y$; now \rminus{pierce} follows by virtue of (\ref{h2}). By the way, we have proved that the Pierce law is equivalent also to the condition (\ref{wm0}).

\section{Commutative wBCK-algebras} \label{commw}

\subsection{Preliminaries}
We extend to weak BCK-algebras also the standard definition of a commutative BCK-algebra \cite{bck1,bck2}.

\begin{defi}
A wBCK-algebras is said to be \emph{commutative} if the identity
\begin{minusenum}
\item  \label{mcommut}
$x - .x - y = y - .y -x$
\end{minusenum}
holds in it.
\end{defi}

\begin{theo}  \label{commiff}
A wBCK-algebra is commutative if and only if it fulfils any of the following (equivalent) conditions:
\begin{minusenum}
\item  \label{midem}
if $x \le y$, then $x \le y - .y - x$,
\item   \label{</-}
$x \le y$ if and only if\/ $y - .y - x = x$,
\item   \label{<//-}
$x \le y$ if and only if $x = y - z$ for some $z$.
\end{minusenum}
\end{theo}
\begin{proof}
Evidently, \rminus{mcommut} implies \rminus{midem} in virtue of \rminus{salg} and \rminus{x0x}. Conversely, $x - .x - y \le y - .y - (x - .x - y) \le y - .y - x$ by \rminus{mdi2}, \rminus{midem} and \rminus{mdi1}, \rminus{manti}. Likewise, $y - .y - x \le x - .x - y$.

The condition \rminus{midem} is included in \rminus{</-}, but is equivalent to the latter by virtue of \rminus{mdi2} and \rminus{xyx}.
Further, the ``if'' parts in both \rminus{</-} and \rminus{<//-} are evident in virtue of \rminus{xyx}, while the ``only if'' part of \rminus{<//-} follows from that of \rminus{</-}. At last, to derive the ``only if'' part of \rminus{</-} from \rminus{<//-}, suppose that $x \le y$. Then $x = y - z$ for some $z$, and $y - .y - x = x$ due to \rminus{mtri}.
\end{proof}

The relationship \rminus{midem} is more convenient than \rminus{mcommut} for checking if a wBCK-algebra is commutative: it does not require looking over so many pairs of elements. For instance, the wBCK-algebra $\Nv$ from Example \ref{exN5} satisfies \rminus{midem}; thus, a commutative wBCK-algebras need not be a BCK-algebra. If we change the last line in its operation table by putting $1 - a = b$, $1 - b = a$, $1 - c = a$, we obtain an non-commutative wBCK-algebra $\Nw$, in which $c \nleq 1 - .1- c$.

We now can extend to commutative wBCK-algebras a result known well for BCK-algebras (see \cite[Theorem 3]{bck2}, where it  was proved using \rminus{exch} essentially).

\begin{coro}   \label{wed/-}
An operation $\wedge$ on a commutative wBCK algebra $A$ is meet in this poset if and only if it is related to subtraction by
\[
x \wedge y = x - .x - y \enspace .
\]
\end{coro}
\begin{proof}
Assume that $\wedge$ is meet. By (\ref{minus}) and \rminus{</-}, then $x - . x - y = x - \di x - .x \wedge y = x \wedge y$. Conversely, according to  \rminus{mdi2} and \rminus{mdi1}, $x - .x - y$ is a lower bound of $x$ and $y$. If $u$ is one more lower bound, then $x - y \le x - u$ and $u = x - .x - u \le x - .x - y$  by \rminus{manti}, \rminus{</-} and \rminus{manti}. Thus $x - .x - y$ is the greatest lower bound of $x$ and $y$, as needed.
\end{proof}

By a different approach, which relies on the order structure of commutative wBCK-algebras, this result is obtained in subsection 6.1 of \cite{wbck3}.
A similar dual construction in a context not related to wBCK*-algebras is presented also in \cite{ChH1}, \label{rf1}\cite{ChHK2}, \cite{ChKS} and \cite{ChL}.
The algebra $\Nw$ just described shows that a wBCK$^\land$-algebra is not necessarily commutative.

\subsection{Equational axioms}

So, every commutative wBCK-algebra can be turned into a wBCK$^\wedge$-algebra, with the meet operation term-definable. This implies that the class of commutative wBCK-algebras is equationally definable (cf.\ subsection \ref{prelim}.3). As shown in \cite{Yut}, the class of commutative BCK-algebras is characterized by equations \rminus{xx0}, \rminus{x0x}, \rminus{mcommut} and \rminus{exch}. An equational axiom system appropriate for commutative wBCK-algebras can be obtained by replacing the latter identity with a weaker one:
 \begin{minusenum}
\item \label{c2}
$z - y. - (z - \di y - .y - x) = 0$,
\end{minusenum}
which is a version of \rminus{manti}.

\begin{theo}    \label{comm:equ}
An algebra $(A,-,0)$ is a commutative wBCK-algebra with respect to the relation $\le$ defined by \textup{\rminus{salg}} if and only if it satisfies the equations {\upshape \rminus{xx0}, \rminus{x0x}, \rminus{mcommut}} and \textup{\rminus{c2}}.
\end{theo}
\begin{proof}
If $A$ is indeed a commutative wBCK-algebra, then \rminus{c2} holds in it by virtue of \rminus{mdi1} and \rminus{manti}. Now assume that $A$ satisfies the four equations. With  Corollary \ref{wbck:eq} in mind, we shall demonstrate (\ref{leanti}), \rminus{0x0}, \rminus{mdi2} and \rminus{manti}.
The relation $\le$ defined by \rminus{salg} is antisymmetric: if $x \le y$ and $y \le x$, then $x = x - 0 = x - .x - y = y - .y - x = y - 0 = y$.
The following particular case of \rminus{c2}:
\[
x - x. - (x - \di x - .x - 0) = 0
\]
leads us, in virtue of \rminus{x0x} and \rminus{xx0}, to the equation $0 - x = 0$,
i.e., to \rminus{0x0}.
Another particular case,
\[
x - .x - y\di - (x - (x - y. - \di x - y. - 0)) = 0,
\]
similarly reduces to the equation $x - .x - y\di - x = 0$, which, together with \rminus{mcommut}, gives us \rminus{mdi2}.
At last, \rminus{manti} holds: if $x \le y$, then $x - y = 0$,  $y - .y - x = x - .x - y = x$ by \rminus{x0x} and $z - y. - .z - x = z - y. - (z - \di y - .y - x) = 0$ by \rminus{c2}, i.e., $z - y \le z - x$.  It remains to apply Corollary \ref{wbck:eq}.
\end{proof}

With some inessential distinctions, a dual set of axioms appears in other connection in Section 6.4 of \cite{ChHK2}.
Equivalent versions of it are used in \label{rf2}
\cite{ChH1}, \cite{ChKS} and \cite{ChL} to define respectively the classes of Abbot groupoids, implication basic algebras and so called strong I-algebras (cf.\ also \cite[Lemma 1]{ChE1}). Therefore, each of these classes may be identified with that of commutative wBCK*-algebras, and various results obtained in the mentioned papers can be transferred to commutative wBCK-algebras.

In fact, the variety of commutative wBCK-algebras is even 3-based.

\begin{theo} \label{comm:equ'}
The class of commutative wBCK-algebras is characterized by axioms \textup{\rminus{mcommut}}, \textup{\rminus{c2}} and either \textup{\rminus{x.0yx}} or
\begin{minusenum}
\item   \label{pw}
$x -\di x - y. - x = x$ \enspace.
\end{minusenum}
\end{theo}
\begin{proof}
Evidently, \rminus{x.0yx} holds in all wBCK-algebras. The identity \rminus{pw} also holds in every wBCK-algebra: due to \rminus{xyx}, \rminus{salg} and \rminus{x0x},
\( x -\di x - y. - x = x - 0 = x\, \).
Now assume that an algebra $(A,-,0)$ satisfies any of the two triples of axioms; we should to demonstrate that they imply \rminus{xx0} and \rminus{x0x} (see Theorem \ref{comm:equ}).

Substituting $0 - y$ for both $x$ and $y$ in \rminus{c2} and using the particular case $0 - y. - .0 - y = 0 - y$ of \rminus{x.0yx}, we obtain the identity $z - .0 - y\di - \di z - .0 - y = 0$. By \rminus{x.0yx}, then $z - z = 0$, which is \rminus{xx0}. In particular, $0 - y. - .0 - y = 0$; this identity together with the mentioned particular case of \rminus{x.0yx} provides \rminus{x0x}.

On the other hand, \rminus{pw} implies that $y - \di y - y. - y = y$, and then it follows from \rminus{c2} that $y - y. - y = 0$. Now  \rminus{pw} implies also \rminus{x0x} and, further, the identity $x - .x - x = x$. Then a substitution of $0$ for $x$ and $y$ in \rminus{c2} yields the identity $z - 0. - .z - 0 = 0$; so, \rminus{xx0} also holds in virtue of \rminus{x0x}.
\end{proof}

\subsection{Commutative qBCK-algebras}
The next proposition implies that the class of commutative qBCK algebras is a subvariety of the variety of commutative wBCK-algebras.

\begin{prop}
A commutative wBCK-algebra is a qBCK-algebra if and only if it satisfies any of the (equivalent) conditions
\begin{minusenum}
\item   \label{miso++}
$y - .y - x\di - z \le y -z$,
\item   \label{miso+}
$y - u. - z \le y - z$.
\end{minusenum}
\end{prop}
\begin{proof}
By \rminus{</-}, \rminus{<//-} and \rminus{xyx}.
\end{proof}

\begin{exam} \label{om6}
This subvariety is proper. The bounded lattice $\Ov$ with six elements $0 < a,b,c,d < 1$ and the operation table
{\small\[
\begin{array}{|c|@{\extracolsep{6pt}}cccccc|}
\hline
\;-\-\, &  0 &  a &  b &  c & d & 1\vphantom{)^{A^A}} \\
\hline
0 & 0 & 0 & 0 & 0 & 0 & 0\vphantom{)^{A^A}} \\
a & a & 0 & a & a & a & 0 \\
b & b & b & 0 & b & b & 0 \\
c & c & c & c & 0 & c & 0 \\
d & d & d & d & d & 0 & 0 \\
1 & 1 & b & a & d & c & 0 \\
\hline
\end{array}
\]}
for an operation $-$ is a commutative wBCK-algebra (in fact, the order dual of the orthoimplication algebra discussed in \cite[Remark]{Abb2}).
As observed there, it does not satisfy the isotonicity law \rminus{miso}: $b \le 1$, $b - c = b$ and $1 - c = 0$.
\par
\end{exam}

\begin{theo}
In a commutative qBCK-algebra, if $x \vee y$ exists, then
\begin{gather}
\mbox{$x \vee y. - x \le y$}.
\end{gather}
\end{theo}
\begin{proof}
Assume that $A$ is a commutative qBCK-algebra, and suppose that $p:= x \vee y$ exists in $A$ for some $x$ and $y$. Since $p - x \le p$, we have that
$p - x. - .p - y \le p - .p - y \le y$. By (\ref{minus}), then $p - x. - (p - x. \wedge .p - y) \le y$. But $p - x. \wedge .p - y \le p - x, p - y$, whence $x,y \le p - (p - x. \wedge .p - y)$ (by \rminus{manti} and \rminus{midem}) and, further, $p \le p - (p - x. \wedge .p - y)$. Thus $(p - x. \wedge .p - y) = p - . p - (p - x. \wedge .p - y) = 0$ (see \rminus{</-} and \rminus{xyx}), and eventually $p - x \le y$.
\end{proof}

\subsection{Uniformity}

Let us consider a collection of equivalent conditions on commutative wBCK-algebras.

\begin{lemm} \label{equi:uni}
The following assertions are equivalent in any commutative wBCK-algebra:
\begin{minusenum}
\item   \label{Ippo`}
if $x \le p \le q$, then $p - x = p - \di p - .q - x$,
\item   \label{compm}
if $x \le p \le q$, then $p - .q - x = x$,
\item   \label{Contr}
if $z \le y$, then $x - z. - y = x - y$,
\item  \label{abb}
$x - .y - z\di - y = x - y$.
\end{minusenum}
\end{lemm}
\begin{proof}
Let $A$ be a commutative wBCK-algebra.

(i) It is easily seen that \rminus{Ippo`} is equivalent to \rminus{compm}. Suppose that $x \le p$. The equation $p - x = p - \di p - .q - x$ implies that $p - (p - x) = p - (p - \di p - .q - x)$, i.e.,  $x = p - .q - x\,$: see \rminus{mtri} and \rminus{</-}. Clearly, the converse implication also holds.

(ii) Further, \rminus{Contr} is equivalent to \rminus{compm}.
Suppose that $z \le y$. By \rminus{manti} and \rminus{xyx}, $x - y \le x - z \le x$. Then \rminus{compm} implies that $x - z. - \di x - .x - y = x - y$. But $x - z. - y \le x - z. - \di x - .x - y$ by \rminus{mdi2} and \rminus{manti}. On the other hand, $x - z. \wedge y \le x \wedge y$, i.e., $x - z. - \di x - z. - y \le x - .x - y$; then the exchange rule \rminus{sgalois} leads us to the reverse
$x - z. - \di x - .x - y \le x - z. - y$ of the inequality just proved.  Eventually, $x - z. - y = x - y$. Now suppose that $x \le p \le q$. Then $q - p \le q - x$ by \rminus{manti}, and then \rminus{Contr} and \rminus{</-} imply that $x = q - .q - x = q - .q - p\di - .q - x =  p - .q - x$.

(iii) In virtue of \rminus{xyx}, the equation \rminus{abb} is a particular case of \rminus{Contr} with $z := y - z$. Conversely, suppose that $z \le y$.
Applying \rminus{</-} and \rminus{abb}, then $(x - z) - y = (x - \di y - .y - z) - y = x - y$.
\end{proof}

Observe that \rminus{Ippo`} can be rewritten as
\begin{equation}    \label{Ippo}
\mbox{if $x \le p \le q$, then $p - x = p \wedge .q - x$} \enspace .
\end{equation}
Let us call a commutative wBCK-algebra \emph{uniform} \cite{wbck3} if it satisfies any of the conditions listed in the lemma.
For instance, the wBCK-algebra $\Ov$ from Example \ref{om6} is uniform.
 Due to \rminus{abb}, the class of uniform commutative wBCK-algebras also is a variety; we shall return to it in Section \ref{uoimpl}.

\section{Orthoimplicative wBCK-algebras}
\label{oimpl}

Theorem 10 in \cite{bck2} says that a BCK-algebra satisfying \rminus{pierce} is both commutative and positively implicative; the converse also holds. Though the BCK-axiom \rminus{mAnti} is used in the proofs of these results, they remain  valid also in wBCK-algebras.

\begin{lemm} \label{posimpl}
The Pierce law \textup{\rminus{pierce}}, the contraction law \textup{\rminus{mcontr}}
and the contraction rule \textup{\rminus{mcontr-}}
are equivalent in commutative wBCK-algebras.
\end{lemm}
\begin{proof}
We already observed in Section \ref{prelim} that \rminus{pierce} implies \rminus{mcontr} and that \rminus{mcontr} implies \rminus{mcontr-}.
It remains to show that \rminus{pierce} follows from \rminus{mcontr-}. Actually, we shall derive \rminus{pierce'}.

So, assume \rminus{mcontr-}, and suppose that $x \le y - x$ (hence, $x \le y$). Then \rminus{manti} implies that $y - .y - x \le  y - x$, whence $y \le y - x$. By \rminus{</-} and \rminus{salg}, now $x = y - .y - x = 0$, as needed.
\end{proof}

\begin{theo}    \label{piercom}
A wBCK-algebra that satisfies the Pierce law is commutative.
\end{theo}
\begin{proof}
Suppose that there is a counterexample---a wBCK-algebra $A$, in which (\ref{wm0}) holds, but \rminus{</-} fails to be true. Then there are elements $p$ and $a$ in $A$ such that $a \le p$, and $p - .p - a < a$. Let $b := p - a$ and $c := p - b$; then $c < a$, $p - c = p - a = b$, $a \wedge b = 0 = b \wedge c$ and $0 < a,b,c < p$. See the black dots in the diagram
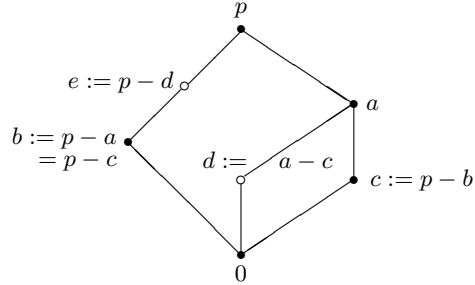
\begin{center}
\setlength\unitlength{.5cm}
\begin{picture}(8,7)(3,2.4) 
\put(7,3){\line(-1,1){3}}
\put(4,6){\line(1,1){1.4}}
\put(5.6,7.55){\line(1,1){1.4}}
\put(7,9){\line(3,-2){3}}
\put(10,7){\line(0,-1){2}}
\put(10,5){\line(-3,-2){3}}
\put(7,3){\line(0,1){1.9}}
\put(7.1,5.1){\line(3,2){2.9}}
\put(7,3){\circle*{.2}}
\put(4,6){\circle*{.2}}
\put(5.5,7.5){\circle{.2}}
\put(7,9){\circle*{.2}}
\put(10,7){\circle*{.2}}
\put(10,5){\circle*{.2}}
\put(7,5){\circle{.2}}
\put(7,2.5){\makebox(0,0){\small $0$}}
\put(3.5,6.1){\makebox(-2.4,0){\small $b := p - a$}}
\put(3.5,5.5){\makebox(-1.7,0){\small $= p - c$}}
\put(5.1,7.6){\makebox(-2.6,0){\small $e := p - d$}}
\put(7,9.5){\makebox(0,0){\small $p$}}
\put(10.5,7){\makebox(0,0){\small $a$}}
\put(10.5,5){\makebox(2.6,0){\small $c := p - b$}}
\put(7,5.6){\makebox(1.4,-0.2){\small $d :=\quad  a - c$}}
\end{picture}
\end{center}

Let, further, $d := a - c$; so, $c \wedge d = 0$ and $0 < d < a$. This implies that,  for $e := p - d$, likewise  $d \wedge e = 0$ and $b < e < p$ (as $p -a \le p - d$). For $f := p - e$, further  $e \wedge f = 0$, $0 < f \le c$ (as $p - e \le p - b$) and $f = p - .p - d \le d$. This contradicts to the above equality $c \wedge d = 0$.
\end{proof}

BCK-algebras satisfying the Pierce law are commonly called implicative; however, we reserve this attribute for more specific wBCK-algebras (see the next section).

\begin{defi}	\label{oi:defi}
A wBCK-algebra is said to be \emph{orthoimplicative} if it satisfies \rminus{pierce}, i.e., is  commutative and positive implicative.
\end{defi}

The algebra $\Ov$ from the Example \ref{om6} is an instance of an orthoimplicative wBCK-algebra that is not a qBCK-algebra and, hence, a BCK-algebra. The next theorem shows that orthoimplicative qBCK-algebras are of little interest for our purposes in this paper.

\begin{theo} \label{oiq}
Every orthoimplicative qBCK-algebra is a BCK-algebra.
\end{theo}
\begin{proof}
By virtue of Theorem \ref{bck(w)}, it suffices to prove that the exchange law \textup{\rminus{exch}} is fulfilled in any  orthoimplicative qBCK-algebra $A$.

Since $A$ satisfies the Pierce law, we conclude from (\ref{wm0}) that $x - y. - z\di \wedge y \le x - y. \wedge y = 0$. Since $A$ is also commutative, further $(x - y. - z) - .(x - y. -z) - y = 0$, i.e., $x - y. - z \le x - y. - z\di - y$. On the other hand, \rminus{xyx} together with \rminus{miso} imply that $x - y. - z\di - y \le x - z. - y$. Thus, $x - y. - z \le x - z. - y$. The reverse inequality follows by symmetry.
\end{proof}

Since commutative wBCK-algebras form a variety, so do also orthoimplicative wBCK-algebras. Notice that \rminus{pw} is a particular case of the Pierce law \rminus{pierce}. Due to Theorem \ref{comm:equ'}, this observation leads us to an economical axiom system for orthoimplicative wBCK-algebras.

\begin{prop}	\label{oi:equ}
An algebra $(A,-,0)$ is an orthoimplicative wBCK-algebra w.r.t.\ to the relation $\le$ defined by \textup{\rminus{salg}} if and only if it satisfies \textup{\rminus{pierce}, \rminus{mcommut}} and \textup{\rminus{c2}}.
\end{prop}

These axioms form a system which, up to minor unessential changes, is dual to the system of axioms for orthoimplication algebras of \cite{Ch4} \label{rf3} and \cite[subsection 6.2.4]{ChHK2} (these differ from orthoimplication algebras of \cite{Abb2}). Ortho-algebras in the sense of \cite{Ch5} is the same class of algebras. Implication orthoalgebras discussed in \cite{ChH2} have some additional axioms, which are in fact redundant. So, all these classes of  algebras may be identified with that of orthoimplicative wBCK*-algebras.

\section{Implicative wBCK-algebras}
\label{uoimpl}

The contraction law \rminus{mcontr} is a particular case (with $z = y$) of \rminus{Contr}. We thus may consider, in the context of commutative wBCK-algebras, uniformity as a strengthening of the property ``being positive implicative''. Lemma \ref{posimpl} then implies that a commutative and uniform wBCK-algebra is orthoimplicative.
Notice also that \rminus{midem} is a consequence of \rminus{compm}; so, a weak BCK algebra is commutative and uniform if and only if it satisfies \textup{\rminus{compm}}.
These observations give rise to the following definition.

\begin{defi}
A wBCK-algebra is said to be
\emph{implicative} if it satisfies \rminus{compm}, i.e., is commutative and uniform. \end{defi}

The next theorem presents several conditions that are necessary and sufficient for an orthoimplicative wBCK-algebra to be implicative; each of the three latter ones is ``a half'' of a condition  from Lemma \ref{equi:uni}. Notice that \rminus{Ippo'} is subsumed also under the isotonicity law \rminus{miso}.

\begin{theo}   \label{semiunif}
A wBCK-algebra is implicative if and only if it is orthoimplicative and satisfies any of the (equivalent) conditions
\begin{minusenum}
\item   \label{Ippo'}
if $x \le p \le q$, then $p - x \le q - x$,
\item   \label{compm'}
if $x \le p \le q$, then $p - .q - x \le x$,
\item   \label{Contr'}
if $z \le y$, then $x - z. - y \le x - y$,
\item    \label{abb'}
$x - .y - z\di - y \le x - y$.
\end{minusenum}
\end{theo}
\begin{proof}
Evidently, \rminus{Ippo'} and \rminus{compm'} are equivalent by \rminus{sgalois}. Further, an inspection of items (ii) and (iii) in the proof of Lemma \ref{equi:uni} shows that the equivalence of \rminus{compm'}--\rminus{abb'} can be proved just in the same way (of course, omitting those portions of the arguments now needless).

We have already noticed that an implicative wBCK-algebra always is orthoimplicative. Further, \rminus{compm'} is included in \rminus{compm}. These remarks end the proof of the ``only if'' part of the theorem. Its "if" part is proved in the subsequent section: see Proposition \ref{suf}.
\end{proof}

Orthoimplicative wBCK-algebras form a proper subclass of implicative wBCK-algebras.

\begin{exam}    \label{oinoti}
Let $A$ be a bounded lattice with five atoms $a,b,c,d,e$ and five coatoms $f,g,h,i,j$ such that $f = e \vee d$, $g = c \vee e$, $h = b \vee d$, $i = a \vee c$, $j = a \vee b$, all other joins of incomparable pairs of elements being equal to 1. (This is the self-dual lattice presented in \cite[Example]{ChH2}.) The operation  $-$ with the table
\begin{center}
\small \(
\begin{array}{|c|@{\extracolsep{6pt}}cccccccccccc|}
\hline
\;-\-\, &  0 &  a &  b &  c & d & e & f & g & h & i & j & 1 \vphantom{)^{A^A}} \\
\hline
0 & 0 & 0 & 0 & 0 & 0 & 0 & 0 & 0 & 0 & 0 & 0 & 0\vphantom{)^{A^A}} \\
a & a & 0 & a & a & a & a & a & a & a & 0 & 0 & 0 \\
b & b & b & 0 & b & b & b & b & b & 0 & b & 0 & 0 \\
c & c & c & c & 0 & c & c & c & 0 & c & 0 & c & 0 \\
d & d & d & d & d & 0 & d & 0 & d & 0 & d & d & 0 \\
e & e & e & e & e & e & 0 & 0 & 0 & e & e & e & 0 \\
f & f & f & f & f & e & d & 0 & d & e & f & f & 0 \\
g & g & g & g & e & g & c & c & 0 & g & e & g & 0 \\
h & h & h & d & h & b & h & b & h & 0 & h & d & 0 \\
i & i & c & i & a & i & i & i & a & i & 0 & c & 0 \\
j & j & b & a & j & j & j & j & j & a & b & 0 & 0 \\
1 & 1 & f & g & h & i & j & a & b & c & d & e & 0 \\
\hline
\end{array} \)
\end{center}
turns it into an orthoimplicative wBCK-algebra (cf.\ Corollary \ref{ortiff}(a)) in which \rminus{Ippo'} fails: $a \le j \le 1$, but $j - a \nleq 1 - a$.
\end{exam}

A class of wBCK-algebras intermediate between orthoimplicative and implicative wBCK-algebras (called semi-implicative wBCK-algebras) will be shortly discussed in subsection \ref{impl:struct}.

Theorem \ref{comm:equ'} and Lemma \ref{equi:uni} provide an equational axiom system for implicative wBCK-algebras consisting of \rminus{mcommut}, \rminus{c2}, \rminus{pw} (or \rminus{x.0yx}) and \rminus{abb}. Another one, \rminus{pierce}, \rminus{mcommut}, \rminus{c2}  and \rminus{abb}, comes from Proposition \ref{oi:equ} and Theorem \ref{semiunif}. The next theorem provides us with a slightly more economic set of axioms.

\begin{theo}
An algebra $(A,-,0)$ is an implicative wBCK-algebra w.r.t. to the relation $\le$ defined by \textup{\rminus{salg}} if and only if it satisfies \textup{\rminus{xx0}, \rminus{pierce}, \rminus{mcommut}} and \textup{\rminus{abb}}.
 \end{theo}
\begin{proof}
Of course, the four listed identities are fulfilled in an implicative wBCK-algebra. On the other hand, the order dual of an algebra $(A,-,0)$ satisfying these identities is essentially an orthoimplication algebra in the sense of \cite{Abb2} (and conversely). By Lemma 1 of that paper, such an algebra $A$ satisfies \rminus{0x0}, \rminus{manti} and (\ref{leanti}). It satisfies also \rminus{mdi1} (put $y = x$ in \rminus{abb}) and, hence, \rminus{mdi2}. So, A is a wBCK-algebra (Corollary \ref{wbck:eq}), which is implicative by definition (see Lemma \ref{equi:uni}).
\end{proof}

Therefore, orthoimplication algebras of \cite{Abb2} (called also orthomodular  implication algebras with the compatibility condition in \cite{ChHK2}), may be identified with implicative wBCK*-algebras. It follows that the algebra $\Ov$ from Example \ref{om6} is an instance of an implicative wBCK-algebra that is not a BCK-algebra (see the note subsequent to Definition \ref{oi:defi}; cf.\ also \cite[Remark]{Abb2}).

\section{Some structure theorems for commutative wBCK-algebras} \label{struct} 

\subsection{Prelimnaries}   \label{struct1}

Theorem 2.3 in \cite{wbck2} discovers the structure of initial segments of wBCK-algebras. In particular, every such a segment of a wBCK-algebra is its subalgebra. Proposition \ref{wbck:char} below is a slightly improved  version of the theorem.

A unary operation $^+$ on a bounded poset is called a \emph{dual Galois complementation} (or just \emph{g*-complementation}) \cite{wbck2} if it satisfies the conditions
    \[ x^{++} \le x,  \quad \text{ if } x \le y, \text{ then } y^+ \le x^+, \quad 0^+ = 1 \]
(then $x^+ = 0$ iff $x = 1$). A poset with $0$ is said to be \emph{sectionally g*-complemented}, if every section $[0,p]$ in it is g*-complemented. Observe that every bounded poset admits the \emph{discrete g*-complementation} defined by
\[
\mbox{$x^+ = 0$ if $x = 1$, and $x^+ = 1$ otherwise}.
\]

\begin{prop}    \label{wbck:char}
Let $A$ be a poset with the least element $0$, a binary operation $-$ and, for every $p \in A$, a unary operation $^+_p$ on $[0,p]$. The following assertions are equivalent:
\begin{myenum}
\thitem{a}
$(A,-,0)$ is a weak BCK-algebra, and $x^+_p = p - x$ for every $p \in A$ and all $x \le p$.
\thitem{b}
Every operation $^+_p$ is a g*-complementation, and
\[
x - y = \min\{z^+_x\colo z \le x,y\} \text{ for all } x,y \in A .
\]
\end{myenum}
If the poset $A$ is a meet semilattice, then the latter condition reduces to
\[
x - y = (x \wedge y)^+_x\; \text{ for all } x,y \in A .
\]
\end{prop}

When speaking on a wBCK-algebra as sectionally g*-complemented, or on a sectionally g*-complemented poset as a wBCK-algebra, we shall have in mind just the operations $^+_p$ defined in (a) and, respectively,  the subtraction described in (b). The equalities (\ref{minus}) and (\ref{Minus}) are easy consequences of this proposition. Notice that there is a bijective correspondence between semilattice ordered wBCK-algebras and sectionally g*-complemented semilattices. (This correspondence is functorial.) Generally, the transfer from wBCK-algebras to sectionally g*-complemented posets is injective, and, of course, different sectionally g*-complemented posets cannot support the same wBCK-algebra. Every sectionally g*-comple\-mented poset with discrete initial segments supports a discrete wBCK-algebra (and conversely). For all that, not every sectionally g*-complemented poset gives raise to a wBCK-algebra.

\begin{exam}
The five-elements poset consisting of four maximal chains $0 < a < c$, $0 < b < c$, $0 < a < d$ and $0 < b < d$ may be regarded as sectionally g*-complemented with $a^+_c = a^+_d = b$ and $b^+_c = b^+_d = a$. This allows us to define $x - y$ uniquely for all values of $x$ and $y$ except for $x = c, y = d$ and $x = d, y = c$. The reason for the exceptions is that neither $a^+_c$ and $b^+_c$ nor $a^+_d$ and $b^+_d$ are comparable.
\end{exam}
 
To proceed, we introduce several particular types of g*-complementation; some of them were discussed in \cite{wbck2}.
A unary operation $^+$ on a bounded poset is said to be
\begin{itemize}
\item
a \emph{semicomplementation} (\emph{s-complementation})
if it satisfies the conditions
\[
\mbox{$x \wedge x^+$ exists and equals $0$, and $x^+ = 0$ only if $x = 1$},
\]
\item
a \emph{dual Brouwerian complementation} (\emph{b*-complementation}) if it is a g*-complementation such that 
 \( x \vee x^+ \text{ exists and equals } 1\)
for all $x$,
 \item
a \emph{De Morgan complementation} (\emph{m-complementation}),
if it is an idempotent g*-complementation,
\item
an \emph{orthocomplementation} (\emph{o-complementation}),
if it is an idempotent b*-complementation; equivalently, if it is both an s- and an m-comple\-mentation. 
\end{itemize}

Let the symbol @ stand for any of symbols s, b*, m, o. A poset with $0$ is said to be \emph{sectionally @-complemented} if every section $[0,p]$ in it is equipped with a  @-complementation.

By Proposition \ref{wbck:char}, a wBCK-algebra $A$, being sectionally g*-complemented, is sectionally b*-complemented if and only if the join of $p - x$ and $x$ in $[0,p]$ exists and equals to $p$ whenever $x \le p$. If $p$ here always turns out to be the join of these elements even in $A$, we shall say that $A$ is \emph{strongly sectionally b*-complemented}. 
The subsequent proposition is essentially a rewording of \cite[Theorem 3.2]{wbck2}. We give it a short independent proof.

\begin{prop} \label{wposiff}
A wBCK-algebra satisfies the contraction rule \textup{\rminus{mcontr-}} if and only if it is strongly sectionally b*-complemented.
\end{prop}
\begin{proof}
Assume that \rminus{mcontr-} is fulfilled in $A$, and suppose that $x \le p$. Clearly, then $p$ is an upper bound of $x$ and $p - x$. If $z$ is one more upper bound, then $p - z \le p - x \le z$ by \rminus{manti}, and, further, $p \le z$ by \rminus{mcontr-}. Therefore, $p$ is the least upper bound of $x$ and $p - x$ in $A$. Thus, $A$ is strongly sectionally b*-complemented.

Conversely, assume that $A$ is strongly sectionally b*-complemented.  Then $x = x - y. \vee x - .x - y \le y$ whenever $x - y \le y$: see \rminus{xyx} and \rminus{mdi2}.
So, \rminus{mcontr-} is valid.
\end{proof}

\subsection{Commutative and orthoimplicative wBCK-algebras}

We can say more about the structure of commutative wBCK-algebras. A meet semilattice in which every pair of elements bounded above has the join is known as a \emph{nearlattice}. It follows that every initial section in a nearlattice is a lattice. Notice that a sectionally b*-compemented wBCK-nearlattice is strongly sectionally b*-complemented. The subsequent proposition is a consequence of \cite[Lemmas 2.5 and 2.2]{wbck2}. 

\begin{prop}   \label{meetnear}
\begin{myenum}
\item A sectionally m-complemented meet semilattice is a nearlattice.
\item A sectionally m-complemented poset is a meet semilattice if and only if it is a wBCK-algebra.
\end{myenum}
\end{prop}

 An m-complemented lattice is called a \emph{non-distributive De Morgan lattice}.

\begin{theo} \label{comiff}
\begin{myenum}
\item
A wBCK-algebra is commutative if and only if its sectional g*-complementations are idempotent, i.e., if it is sectionally m-complemented.
\item
A commutative wBCK-algebra is a nearlattice in which every section is a non-distributive De Morgan lattice (and conversely).
\end{myenum}
\end{theo}
\begin{proof}
(a) Due to Proposition \ref{wbck:char}, the operations $^+_p$ in a wBCK-algebra are idempotent, i.e., are m-complemen\-tations, just in the case when \rminus{midem} holds.

(b) The mentioned properties of a commutative wBCK-algebra follow immediately from definitions, (a), Corollary \ref{wed/-} and Proposition \ref{meetnear}(a). Conversely, a nearlattice with non-distributive De Morgan sections is sectionally m-complemented, hence, a commutative wBCK-algebra by (a) and Proposition \ref{meetnear}(b).
\end{proof}

Order duals of such structures (without relating them to wBCK*-alge\-bras, of course) have been discussed already in \cite{Ch0} and the papers mentioned in Section \ref{commw} (pp.\ \pageref{rf1} and \pageref{rf2}). Cf.\ also Corollary 6.4 in \cite{wbck3}. 

It is now known well that a bounded commutative BCK-algebra can be equipped with the structure of MV-algebra; even more, classes of bounded BCK-algebras and MV-algebras are term-equivalent \cite{COM}. In \cite{ChKS}, a similar equivalence is stated for bounded implicative basic algebras (i.e., bounded commutative wBCK*-algebras, see Section \ref{commw}) and a version of non-asso\-ciative MV-algebras known as basic algebras \cite{ChHK2,ChHK3}.
Lattices studied in \cite{Ch3,ChE1} in connection with MV-algebras also are in fact bounded commutative wBCK*-algebras, while lattices from an earlier paper \cite{ChHK1} are essentially even (bounded commutative) BCK*-algebras.
The reader could derive from \cite[Theorem 2]{Ch3} necessary and sufficient conditions for sections of a bounded commutative wBCK-algebra to be distributive De Morgan lattices.

By definition, an \emph{ortholattice} is an o-complemented lattice.
\begin{coro} \label{ortiff}
\begin{myenum}
\item
A wBCK-algebra is orthoimplicative if and only if it is sectionally s-complemented.
\item
An orthoimplicative wBCK-algebra is a nearlattice in which every section is an ortholattice (and conversely).
\end{myenum}
\end{coro}
\begin{proof}
(a) Let $A$ be a wBCK-algebra. By Proposition \ref{wbck:char}, it (being sectionally g*-complemented) is  sectionally s-complemented if and only if $x \wedge .p - x = 0$ whenever $x \le p$, i.e., if and only if $p \wedge x. \wedge \di p - .p \wedge x = 0$. By \rminus{xyx} and (\ref{minus}), the latter condition is equivalent to the identity $x \wedge. p - x = 0$, which is (\ref{wm0}), and, hence, to the Pierce law, as needed

(b) By Lemma \ref{piercom}, an orthoimplicative wBCK-algebra is commutative. Now use the preceding theorem, Proposition \ref{wposiff} and Lemma \ref{posimpl}.
\end{proof}

For related dual structures, see the papers mentioned on p.\ \pageref{rf3} in Section \ref{oimpl}. Nearlattices appearing in (b) have been called orthosemilattices in \cite{Ch2}.
 The above corollary has the following immediate consequence.

 \begin{coro}    \label{s-o}
 Every sectionally s-complemented wBCK-algebra is sectionally o-complemented (and conversely).
 \end{coro}

 Sectionally s-complemented  distributive nearlattices have been studied in \cite{NR}. As shown in \cite{JoW}, a sectionally s-complemented poset is distributive if and only if it is 0-distributive in the sense explained in that paper. It follows that an orthoimplicative wBCK-algebra has Boolean sections (and, hence, is a(n implicative) BCK-algebra) if and only if it is 0-distributive.

\subsection{Implicative wBCK-algebras}\label{impl:struct}

It is known that implication algebras of \cite{Abb1a,Abb1b} (known also as Tarski algebras) coincide with implicative BCK*-algebras; see, e.g., \cite{Jie}. So, any implicative BCK-algebra is a nearlattice with Boolean sections. This result has the counterpart for weak BCK-algebras, with orthomodular lattices instead of Boolean ones. Recall that an orthomodular lattice is an ortholattice in which $y = x \vee .y \wedge x^+$ whenever $x \le y$. Equivalently \cite{B}, an ortholattice is orthomodular if $y = x$ in it whenever
$x \le y$ and $x^+ \wedge y = 0$.

\begin{theo} \label{impiff1}
\begin{myenum}
\item
A wBCK-algebra is implicative if and only if it is sectionally m-complemented and satisfies the condition
\begin{equation}    \label{ippo}
\mbox{if $x \le p \le q$, then $x^+_p = p \wedge x^+_q$}.
\end{equation}
\item
An implicative wBCK-algebra is a nearlattice in which every section is an orthomodular lattice.
\end{myenum}
\end{theo}
\begin{proof}
(a)
By Theorems \ref{comiff}(a) and \ref{wbck:char}, condition  (\ref{Ippo}), and Lemma \ref{equi:uni}.

(b) An implicative wBCK-algebra is, in particular,  orthoimplicative. By Theorem \ref{ortiff}(b), it is a nearlattice with  initial sections ortholattices. Now suppose that $x \le y \le q$. Then $y = x \vee x^+_y$, and we, applying (\ref{ippo}), come to the equality $y = x \vee (y \wedge x^+_q)$. Consequently, the ortholattice $[0,q]$ is orthomodular.
\end{proof}

As to (a), notice that a sectionally m-complemented poset satisfying (\ref{ippo}) may actually be called \emph{relatively m-complemented}; so, a wBCK-algebra is implicative if and only if it is relatively m-complemented. In connection with (b), cf.\ Theorem 4 in \cite{ChE3} or Proposition in \cite{Ch2} and their proofs. Semi-orthomodular lattices of  \cite{Abb2} are order duals of  those sectionally orthomodular meet semilattices that satisfy (\ref{ippo}); so, they can be characterized also as implicative wBCK*-algebras (see the connection between wBCK-algebras, nearlattices and  sectionally m-complemented posets stated in Proposition \ref{meetnear}). Further, a generalized orthomodular lattice \cite{Jan,B} can be defined as a sectionally orthomodular lattice satisfying (\ref{ippo}); see \cite{wbck3} for more details. We conclude that a wBCK-lattice is implicative if and only if it is a generalized orthomodular lattice.

It should be noted that the converse of (b) does not hold true: not every sectionally orthomodular nearlattice (alias an orthomodular semilattice \cite{Ch2}) satisfies  (\ref{ippo}): see \cite[Remark]{Ch2}. Correspondingly, not every sectionally orthomodular wBCK-algebra is implicative. Let us call such  wBCK-algebras \emph{semi-implicative}. It follows from Theorems 3--5 of \cite{ChHL1} (and Proposition \ref{meetnear} above) that a wBCK-algebra is semi-implicative if and only if it is the order dual of an orthomodular implication algebra in the sense of \cite{ChHL1,ChHL2,ChL}.
 
 Therefore, every  semi-implicative wBCK-algebra is orthoimplicative, and every implicative wBCK-algebra is  semi-implicative. Of course, the orthoimplicative wBCK-algebra from Example \ref{oinoti} is not semi-implicative; cf.\ \cite[Section 4]{ChHL2}. In its turn, Theorem 4.2 in \cite{MP} shows an orthomodular implication algebra that is not an orthoimplication algebra in the sense of \cite{Abb2}. In terms of the present paper, it is a semi-implicative wBCK*-algebra which is not implicative  (cf.\ the note at the end of Section \ref{uoimpl}).

In some situations, however, the requirement (\ref{ippo}) in Theorem \ref{impiff1}(a) can be weakened.

\begin{theo} \label{semiunif'}
A wBCK-algebra is implicative if and only if it is sectionally s-complemented and satisfies the condition
\begin{equation}    \label{ippo'}
\mbox{if $x \le p \le q$, then $x^+_p  \le x^+_q$},
\end{equation}
\end{theo}
\begin{proof}
Let $A$ be some wBCK-algebra. If it is implicative, then, being orthoimplicative, it is sectionally s-complemented by Corollary \ref{ortiff}(a). Moreover, \rminus{Ippo'} implies (\ref{ippo'}).

Now assume that A is sectionally s-complemented  and satisfies (\ref{ippo'}).
Then it is a sectionally orthocomplemented nearlattice (Corollary \ref{ortiff}), 
and we may use De Morgan duality laws in every section. Suppose that  $x \le p \le q$. Then $x^+_p \le p$ and
$x = (x^+_p)^+_p \le (x^+_p)^+_q$. Put $z := (x^+_p)^+_q$; clearly, $x \le z \le q$. As $x^+_z \le z$ and $x^+_z \le x^+_q$, further
$z = x \vee x^+_z \le x \vee .z \wedge x^+_q \le z$, wherefrom  $z = x \vee .z \wedge x^+_q$.  Now
$x^+_p = z^+_q = x^+_q \wedge .z^+_q \vee x = x^+_q \wedge .x^+_p \vee x = x^+_q \wedge p$, i.e., (\ref{ippo}) also is fulfilled. By the previous theorem, $A$ is implicative.
\end{proof}

By Proposition \ref{wbck:char}, the conditions (\ref{ippo'}) and \rminus{Ippo'} are equivalent. Then, in view of Corollary \ref{ortiff}, the above theorem implies, in particular, the following result, which completes the proof of Theorem \ref{semiunif}.

\begin{prop}    \label{suf}
An orthoimplicative wBCK-algebra satisfying \textup{\rminus{Ippo'}} is implicative.
\end{prop}

\end{document}